\documentclass[11pt, onecolumn, fleqn]{article}

\usepackage{amsmath}
\usepackage{amsfonts}
\usepackage{amssymb}
\usepackage{amsthm}
\usepackage{enumitem}
\usepackage{graphicx}
\usepackage[hang,flushmargin]{footmisc}
\usepackage[paper={letterpaper}, margin={1.25in}, marginratio={1:1}]{geometry}
\usepackage{hyperref}
\usepackage[latin1]{inputenc}
\usepackage{makeidx}
    \makeindex
\usepackage{mathtools}
\usepackage[framemethod=TikZ]{mdframed}
\usepackage{multicol}
\usepackage{rotating}
\usepackage{stmaryrd}
\usepackage{tikz}
    \tikzset{node distance=3cm and 5cm, auto}
    \usetikzlibrary{arrows}
\usepackage{tikz-cd}
\usepackage{xcolor}

\makeatletter
\def\thm@space@setup{%
  \thm@preskip=\parskip
  \thm@postskip=0pt
}
\makeatother

\theoremstyle{definition}
\newtheorem{theorem}{Theorem}[section]
\newtheorem*{theorem*}{Theorem}

\newtheorem*{corollary*}{Corollary}
\newtheorem{lemma}[theorem]{Lemma}
\newtheorem*{lemma*}{Lemma}
\newtheorem{proposition}[theorem]{Proposition}
\newtheorem*{proposition*}{Proposition}
\newtheorem{conjecture}[theorem]{Conjecture}
\newtheorem*{conjecture*}{Conjecture}
\newtheorem{definition}[theorem]{Definition}
\newtheorem*{definition*}{Definition}

\newtheorem*{notation*}{Notation}
\newtheorem{remark}[theorem]{Remark}
\newtheorem*{remark*}{Remark}
\newtheorem{example}[theorem]{Example}
\newtheorem*{example*}{Example}

\definecolor{fncol}{HTML}{666600}

\newcommand{\slice}[1]{/_{\hspace{-1pt}#1}}
\newcommand{\pullback}{\arrow[dr, draw=none, "\text{\Large$\lrcorner$}" description, pos=0.1]}
\newcommand{\pullbackc}[3][0]{\arrow[#2, draw=none, "\text{\rotatebox{#1}{\Large$\lrcorner$}}" description, pos=#3]}

\newcommand{\seqbn}[2]{\left(\left. #1 \ \middle\rvert\ #2 \right.\right)}

\let\nsum\sum
\let\nprod\prod
\renewcommand{\sum}{\nsum\limits}
\renewcommand{\prod}{\nprod\limits}

\newcommand{\declpoly}[7]{#1 \xleftarrow{#5} #2 \xrightarrow{#6} #3 \xrightarrow{#7} #4}

\newcommand{\pto}{%
  \mathrel{\ooalign{\hfil$\vcenter{
    \hbox{\scalebox{0.7}{$\scriptscriptstyle|\,$}}}$\hfil\cr$\to$\cr}
  }%
}
\newcommand{\pRightarrow}{%
  \mathrel{\ooalign{\hfil$\vcenter{
    \hbox{\scalebox{0.7}{$\scriptscriptstyle|\,$}}}$\hfil\cr$\Rightarrow$\cr}
  }%
}
\newcommand{\pRrightarrow}{%
  \mathrel{\ooalign{\hfil$\vcenter{
    \hbox{\scalebox{0.7}{$\scriptscriptstyle|\,$}}}$\hfil\cr$\Rrightarrow$\cr}
  }%
}

\newcommand{\cart}{^{\mathrm{cart}}}

\setlist{topsep=0pt}

\newenvironment{enumabc}{\begin{enumerate}[label={(\alph*)}]}{\end{enumerate}}
\newenvironment{enumrom}{\begin{enumerate}[label={(\roman*)}]}{\end{enumerate}}

\numberwithin{equation}{section}

\newcommand{\twocell}[1]{\text{\rotatebox[origin=c]{#1}{$\Rightarrow$}}}
\newcommand{\threecell}[1]{\text{\rotatebox[origin=c]{#1}{$\Rrightarrow$}}}

\usepackage{datetime}
\newdateformat{customdate}{\dayofweekname{\THEDAY}{\THEMONTH}{\THEYEAR} \ordinaldate{\THEDAY} \monthname[\THEMONTH] \THEYEAR}

\newcommand{\nameof}[1]{\left[ #1 \right]}

\newcommand{\eqpunct}[1]{#1}

\title{Polynomial pseudomonads and dependent type theory}
\author{%
    \begin{minipage}[t]{180pt} \centering
       {\scshape Steve Awodey}\\
       ~\vspace{-7pt}\\
       {\small Departments of Philosophy\\
       and Mathematical Sciences\\
       Carnegie Mellon University}
    \end{minipage}
    \begin{minipage}[t]{180pt} \centering
       {\scshape Clive Newstead}\\
       ~\vspace{-7pt}\\
       {\small Department of Mathematical Sciences\\
       Carnegie Mellon University}
    \end{minipage}
}
\date{\customdate\today}

\begin{document}

\maketitle

\thispagestyle{empty}
\begin{abstract} \noindent
We assemble polynomials in a locally cartesian closed category into a tricategory, allowing us to define the notion of a \textit{polynomial pseudomonad} and \textit{polynomial pseudoalgebra}. Working in the context of \textit{natural models} of type theory, we prove that dependent type theories admitting a unit type and dependent sum types give rise to polynomial pseudomonads, and that those admitting dependent product types give rise to polynomial pseudoalgebras.
\end{abstract}

\tableofcontents

\newpage
\section{Review of polynomial monads}
\label{secPolynomialMonads}

The material reviewed in this section can be found entirely in \cite{GambinoKock2013}, with the exception of Lemma \ref{lemPolynomialsFromOneToOne}.

Recall that locally cartesian closed categories $\mathcal{E}$ are characterised by the fact that every morphism $f : B \to A$ induces a triple of adjoint functors
\begin{center}
\begin{tikzcd}[column sep=huge, row sep=huge]
\mathcal{E} \slice{A}
\arrow[rr, "\Delta_f" description, ""{name=codtop}, ""'{name=dombot}]
&&
\mathcal{E} \slice{B}
\arrow[ll, bend right, "\Sigma_f"', ""{name=domtop}]
\arrow[ll, bend left, "\Pi_f", ""'{name=codbot}]
\arrow[draw=none, from=domtop, to=codtop, inner sep={1pt}, "\bot" description]
\arrow[draw=none, from=dombot, to=codbot, inner sep={1pt}, "\bot" description]
\end{tikzcd}
\eqpunct{,}
\end{center}
where $\Sigma_f$ is given by postcomposition with $f$ and $\Delta_f$ is given by pullback along $f$. This condition is equivalent to the assertion that all slices of $\mathcal{E}$ are cartesian closed. We adopt the convention that locally cartesian closed categories have a terminal object, so in particular they have all finite limits.

We emphasise that locally cartesian closed categories are categories with additional \textit{structure}. In particular, given a morphism $x : X \to A$, the functor $\Delta_f : \mathcal{E}\slice{A} \to \mathcal{E}\slice{B}$ gives a \textit{choice} of pullback $\Delta_f(x) : \Delta_f X \to B$ of $x$ along $f$.

Examples of locally cartesian closed categories include the category $\mathbf{Set}$ of sets, the category $\widehat{\mathbb{C}}=\mathbf{Set}^{\mathbb{C}^{\mathrm{op}}}$ of presheaves on a small category $\mathbb{C}$, and more generally, any topos.

In what follows, unless otherwise specified, $\mathcal{E}$ will be a fixed locally cartesian closed category.

\begin{definition}
\label{defPolynomial}
A \textbf{polynomial} $F=(s,f,t)$ from $I$ to $J$ in $\mathcal{E}$ is a diagram of the form
\begin{center}
\begin{tikzcd}[row sep=small, column sep=small]
&
B
\arrow[dl,"s"']
\arrow[r,"f"]
&
A
\arrow[dr,"t"]
&
\\
I
\arrow[rrr, draw=none, purple, "F" description]
&&&
J
\end{tikzcd}
\eqpunct{.}
\end{center}
We will write $F : I \pto J$ to denote the assertion that $F$ is a polynomial from $I$ to $J$ in $\mathcal{E}$, and if we wish to be more precise, we will write $F : \declpoly IBAJsft$ to specify all the data.
\end{definition}

\begin{example}
\label{exExamplesOfPolynomials}
Some examples of polynomials include:
\begin{enumabc}
\item Every morphism $f : B \to A$ can be considered as a polynomial $1 \pto 1$, by taking $s,t$ to be the unique morphisms to the terminal object. In the following, we will blur the distinction between morphisms of $\mathcal{E}$ and polynomials $1 \pto 1$ in $\mathcal{E}$.
\item Given an object $I$, the \textbf{identity polynomial} on $I$ is given by $i_I : \declpoly IIII{\mathrm{id}_I}{\mathrm{id}_I}{\mathrm{id}_I}$.
\item Every span $I \xleftarrow{s} A \xrightarrow{t} J$ induces a polynomial $\declpoly IAAJs{\mathrm{id}_A}t$. Polynomials induced by spans are called \textbf{linear polynomials}.
\end{enumabc}
\end{example}

\begin{definition}
\label{defExtension}
\label{defPolynomialFunctor}
Let $F : \declpoly IBAJsft$ be a polynomial in $\mathcal{E}$. The \textbf{extension} of $F$ is the functor
$$P_F = \Sigma_t \Pi_f \Delta_s : \mathcal{E} \slice{I} \to \mathcal{E} \slice{J}
\eqpunct{.}$$
A \textbf{polynomial functor} is one that is naturally isomorphic to the extension of a polynomial.
\end{definition}

We can describe the extension $P_F$ of a polynomial $F : \declpoly IBAJs{}{}$ in the internal language of $\mathcal{E}$ by
$$P_F(X_i \mid i \in I) = \seqbn{\sum_{a \in A_j} \prod_{b \in B_a} X_{s(b)}}{j \in J}
\eqpunct{.}$$

\begin{example}
The extensions of the polynomials in Example \ref{exExamplesOfPolynomials} are given by the following:
\begin{enumabc}
\item If $f : B \to A$ is a morphism of $\mathcal{E}$, considered as a polynomial $1 \pto 1$, then $P_f : \mathcal{E} \to \mathcal{E}$ is a \textbf{polynomial endofunctor}, defined in the internal language by $P_f(X) = \sum_{a \in A} X^{B_a}$.
\item The extension of the identity polynomial $i_I : I \pto I$ is the identity functor $\mathcal{E} \slice{I} \to \mathcal{E} \slice{I}$.
\item The extension of a linear polynomial $F : \declpoly IAAJs{\mathrm{id}_A}t$ is the functor $P_F : \mathcal{E}\slice{I} \to \mathcal{E}\slice{J}$ defined by $P_F(X_i \mid i \in I) = \seqbn{\sum_{a \in A} X_{s(a)}}{j \in J}$. This justifies the term `linear'.
\end{enumabc}
\end{example}

\begin{definition}
\label{defPolynomialComposition}
\index{polynomial composition}
Let $F : \declpoly IBAJsft$ and $G : \declpoly JDCKugv$ be polynomials in $\mathcal{E}$. The \textbf{polynomial composite} of $G$ with $F$ is the polynomial $G \cdot F : I \pto K$ as in
\begin{center}
\begin{tikzcd}
&
N
\arrow[dl,"s \circ n"']
\arrow[r,"q \circ p"]
&
M
\arrow[dr,"v \circ w"]
&
\\
I
\arrow[rrr, draw=none, purple, "G \cdot F" description]
&&&
K
\end{tikzcd}
\eqpunct{.}
\end{center}
where $n,p,q,w$ are defined as in Figure \ref{figPolynomialComposition}, in which {\color{blue} (1)} is a pullback square, $w = \Pi_g(h)$, {\color{blue} (2)} is a pullback square, $e$ is the component at $h$ of the counit of the adjunction $\Delta_g \dashv \Pi_g$, and {\color{blue} (3)} is a pullback square.
\begin{figure}[ht]
\centering
\begin{tikzcd}[row sep=large, column sep=large]
&&&
N
\pullbackc{d}{0.1}
\arrow[ddll,"n"']
\arrow[r,"p"]
&
\bullet
\pullback
\arrow[r,"q"]
\arrow[dd]
\arrow[dl,"e" description]
\arrow[ddr,blue,"\text{(2)}" description,draw=none]
\arrow[ddlll,blue,"\text{(3)}" description,draw=none]
&
M
\arrow[dd,"w"]
&
\\
&&&
\bullet
\pullbackc[-45]{dd}{0.05}
\arrow[dr,"h" description]
\arrow[dl]
\arrow[dd,blue,"\text{(1)}" description,draw=none]
&&{~}&
\\
&
B
\arrow[r,"f"']
\arrow[dl,"s"']
&
A
\arrow[dr,"t"']
&&
D
\arrow[dl,"u"]
\arrow[r,"g"']
&
C
\arrow[dr,"v"]
&
\\
I
&&&
J
&&&
K
\\
\end{tikzcd}
\caption{Construction of the polynomial composite of $F$ with $G$.}
\label{figPolynomialComposition}
\end{figure}
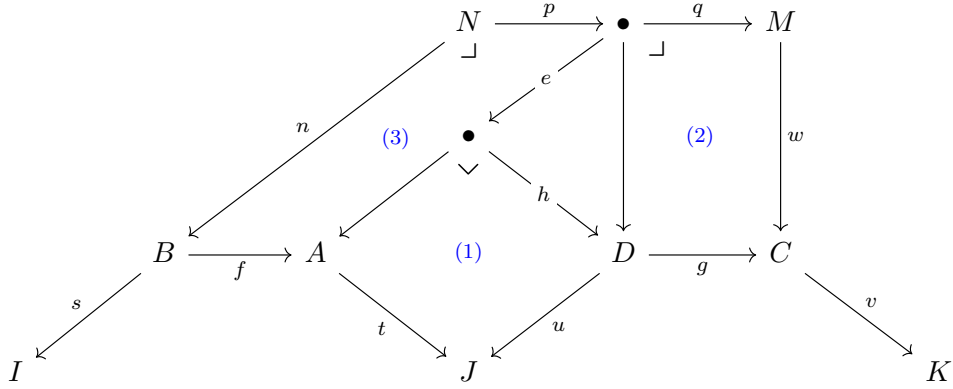
\end{definition}

\begin{remark}
\label{rmkPolynomialCompositionExplicit}
As explained in \cite{GambinoKock2013}, in the internal language of $\mathcal{E}$, we have
$$(M \xrightarrow{q \circ p} K) = \seqbn{\sum_{c \in C_k} \prod_{d \in D_c} A_{u(d)}}{k \in K} \text{ and } (N \xrightarrow{v \circ w} M) = \seqbn{\sum_{d \in D_c} B_{m(d)}}{(c,m) \in M}
\eqpunct{,}$$
so that, in full detail, we can write
$$P_{G \cdot F}(X) = \seqbn{\sum_{(c,m) \in \left(\sum_{c \in C_k} \prod_{d \in D_c} A_{u(d)}\right)} \prod_{(d,b) \in \left(\sum_{d \in D_c} B_{m(d)}\right)} X_{u(d)}}{k \in K}
\eqpunct{.}$$
\end{remark}

This definition of composition of polynomials is motivated by the following.

\begin{proposition}[\textit{Extension preserves composition of polynomials}]
\label{propExtensionPreservesPolynomialComposition}
Let $F : I \pto J$ and $G : J \pto K$ be polynomials in $\mathcal{E}$. There is a natural isomorphism
$$P_{G \cdot F} \cong P_G \circ P_F : \mathcal{E} \slice{I} \to \mathcal{E} \slice{K}
\eqpunct{.}$$
\end{proposition}

\begin{definition}
\label{defMorphismOfPolynomials}
\index{morphism of polynomias}
Let $F : \declpoly IBAJsft$ and $G : \declpoly IDCJugv$ be polynomials from $I$ to $J$ in $\mathcal{E}$. A \textbf{morphism of polynomials} $\varphi$ from $F$ to $G$ consists of an object $D_{\varphi}$ of $\mathcal{E}$ and a triple $(\varphi_0, \varphi_1, \varphi_2)$ of morphisms in $\mathcal{E}$ fitting into a commutative diagram of the following form, in which the lower square is a pullback:
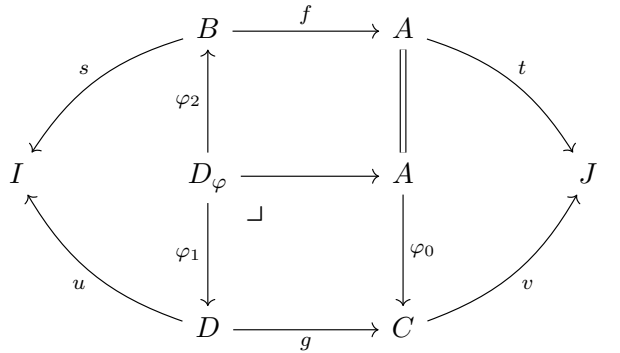
\begin{center}
\begin{tikzcd}[row sep=huge, column sep=huge]
&
B
\arrow[r,"f"]
\arrow[dl, bend right=20, "s"']
&
A
\arrow[d, equals]
\arrow[dr, bend left=20, "t"]
&
\\
I
&
D_{\varphi}
\pullback
\arrow[r]
\arrow[d,"\varphi_1"']
\arrow[u,"\varphi_2"]
&
A
\arrow[d,"\varphi_0"]
&
J
\\
&
D
\arrow[r,"g"']
\arrow[ul, bend left=20, "u"]
&
C
\arrow[ur, bend right=20, "v"']
&
\end{tikzcd}
\eqpunct{.}
\end{center}
We write $\varphi : F \pRightarrow G$ to denote the assertion that $\varphi$ is a morphism of polynomials from $F$ to $G$.
\end{definition}

Each morphism $\varphi : F \pRightarrow G$ of polynomials induces a strong\footnote{Every polynomial functor has a natural \textit{strength}, and the natural candidate for morphisms between polynomial functors are those natural transformations which are comptable with the strength. See \cite{GambinoKock2013} for more on this.} natural transformation $P_F \Rightarrow P_G$, which we shall by abuse of notation also call $\varphi$, whose component at $\vec X = (X_i \mid i \in I)$ can be expressed in the internal language of $\mathcal{E}$ by
$$(\varphi_{\vec X})_j : \sum_{a \in A_j} \prod_{b \in B_a} X_{s(b)} \to \sum_{c \in C_j} \prod_{d \in D_c} X_{u(b)}; \quad (\varphi_{\vec X})_j(a,t) = (\varphi_0(a), t \cdot (\varphi_2)_a \cdot (\varphi_1)_a^{-1})
\eqpunct{.}$$

\begin{definition}
\label{defCartesianMorphism}
A morphism $\varphi : F \pRightarrow G$ is \textbf{cartesian} if $\varphi_2$ is invertible.
\end{definition}

As the name suggests, if $\varphi : F \pRightarrow G$ is a cartesian morphism, then the induced strong natural transformation $P_F \Rightarrow P_G$ is cartesian.

\begin{remark}
\label{rmkCartesianMorphismIsPullbackSquare}
Every cartesian morphism of polynomials has a unique representation as a commutative diagram of the following form:

\begin{equation} \label{diaCartesianMorphism}
\hspace{72pt}
\begin{tikzcd}[row sep=normal, column sep=huge]
&
B
\arrow[dl, bend right=15, "s"']
\arrow[r, "f"]
\arrow[dd, "\varphi_1"']
\pullbackc{ddr}{0.05}
&
A
\arrow[dd, "\varphi_0"]
\arrow[dr, bend left=15, "t"]
&
\\
I
&&&
J
\\
&
D
\arrow[ul, bend left=15, "u"]
\arrow[r, "g"']
&
C
\arrow[ur, bend right=15, "v"']
&
\end{tikzcd}
\eqpunct{.}
\end{equation}

Indeed, if $(\varphi_0,\varphi_1,\varphi_2)$ is cartesian, replacing $\varphi_1$ in the above diagram by $\varphi_1 \circ \varphi_2^{-1}$ yields the desired diagram. Conversely, if $(\varphi_0,\varphi_1)$ are as in the above diagram, then $(\varphi_0,\varphi'_1,\varphi'_2)$ is a cartesian morphism of polynomials, where $\varphi'_1 : \Delta_{\varphi_0}D \to D$ is the chosen pullback of $\varphi_0$ along $g$ and $\varphi'_2 : \Delta_{\varphi_0}D \to B$ is the canonical isomorphism induced by the universal property of pullbacks, as illustrated in the following:
\begin{equation} \label{diaCartesianMorphismFromPullbackSquare}
\begin{tikzcd}[row sep=normal, column sep=normal]
&
B
\arrow[r,"f"]
\arrow[dl, bend right=20, "s"']
\arrow[dd, "\varphi_1"']
\pullback
&
A
\arrow[dr, bend left=20, "t"]
\arrow[dd, "\varphi_0"]
&
&& 
&
B
\arrow[r,"f"]
\arrow[dl, bend right=20, "s"']
&
A
\arrow[d, equals]
\arrow[dr, bend left=20, "t"]
&
\\
I
&&~&
J
&=& 
I
&
\Delta_{\varphi_0}D
\pullbackc{dr}{-0.05}
\arrow[r]
\arrow[d,"\varphi'_1"']
\arrow[u,"\varphi'_2", "\scriptsize\cong"']
&
A
\arrow[d,"\varphi_0"]
&
J
\\
&
D
\arrow[r,"g"']
\arrow[ul, bend left=20, "u"]
&
C
\arrow[ur, bend right=20, "v"']
&
&& 
&
D
\arrow[r,"g"']
\arrow[ul, bend left=20, "u"]
&
C
\arrow[ur, bend right=20, "v"']
&
\end{tikzcd}
\eqpunct{.}
\end{equation}

Note that, in general, for each diagram of the form \eqref{diaCartesianMorphism}, there are possibly many cartesian morphisms inducing it. Conversely, there are many potential ways of turning a diagram of the form \eqref{diaCartesianMorphism} into a cartesian morphism. Another possibility would be to take the induced cartesian morphism to be $(\varphi_0,\varphi_1,\mathrm{id}_B)$. Theorem \ref{thmPolyECartTrivial} below implies that these are essentially equivalent.

In particular, when $I=J=1$, we can regard pullback squares as cartesian morphisms in a canonical way.
\end{remark}

We are now ready to assemble polynomials into a bicategory (and polynomial functors into a $2$-category). In fact, as proved in \cite{GambinoKock2013}, more is true:

\begin{theorem}
\label{thmPolyEBicategory}
Let $\mathcal{E}$ be a locally cartesian closed category.
\begin{enumabc}
\item There is a bicategory $\mathbf{Poly}_{\mathcal{E}}$ whose 0-cells are the objects of $\mathcal{E}$, whose 1-cells are polynomials in $\mathcal{E}$, and whose 2-cells are morphisms of polynomials.
\item There is a $2$-category $\mathbf{PolyFun}_{\mathcal{E}}$ whose 0-cells are the slices $\mathcal{E} \slice{I}$ of $\mathcal{E}$, whose 1-cells are polynomial functors, and whose 2-cells are strong natural transformations.
\item Extension defines a biequivalence $\mathrm{Ext} : \mathbf{Poly}_{\mathcal{E}} \xrightarrow{\simeq} \mathbf{PolyFun}_{\mathcal{E}}$.
\item Parts (a)--(c) hold true if we restrict the 1-cells to \textit{cartesian} morphisms of polynomials in $\mathbf{Poly}_{\mathcal{E}}$ and \textit{cartesian} strong natural transformations in $\mathbf{PolyFun}_{\mathcal{E}}$; thus there is a bicategory $\mathbf{Poly}\cart_{\mathcal{E}}$ and a $2$-category $\mathbf{PolyFun}\cart_{\mathcal{E}}$, which are biequivalent.
\end{enumabc}
\end{theorem}

\begin{definition}
\label{defPolynomialMonad}
A \textbf{polynomial monad} is a monad in the bicategory $\mathbf{Poly}\cart_{\mathcal{E}}$. Specifically, a polynomial monad is a quadruple $\mathbb{P} = (I,p,\eta,\mu)$ consisting of an object $I$ of $\mathcal{E}$, a polynomial $p : I \pto I$ in $\mathcal{E}$ and cartesian morphisms of polynomials $\eta : i_1 \pRightarrow p$ and $\mu : p \cdot p \pRightarrow p$, satisfying the usual monad axioms, namely
$$\mu \circ (\mu \cdot p) = \mu \circ (p \cdot \mu) \quad \text{and} \quad \mu \circ (\eta \cdot p) = \mathrm{id}_p = \mu \circ (p \cdot \eta)
\eqpunct{.}$$
\end{definition}

\begin{remark}
\label{rmkPolynomialMonads}
What is usually (e.g.\ \cite{GambinoKock2013}) meant by a \textit{polynomial monad} is a monad $(P,\eta,\mu)$ on a slice $\mathcal{E}\slice{I}$ of $\mathcal{E}$, with $P : \mathcal{E}\slice{I} \to \mathcal{E}\slice{I}$ a polynomial functor and $\eta,\mu$ cartesian natural transformations; equivalently, this is a monad in the $2$-category $\mathbf{PolyFun}\cart_{\mathcal{E}}$. We recover this notion from Definition \ref{defPolynomialMonad} by applying the extension bifunctor $\mathbf{Poly}\cart_{\mathcal{E}} \to \mathbf{PolyFun}\cart_{\mathcal{E}}$. Furthermore, every polynomial monad in the usual sense is the extension of a polynomial monad in the sense of Definition \ref{defPolynomialMonad}.
\end{remark}

Before we continue, the following technical lemma will simplify matters for us greatly down the road, as it allows us in most instances to prove results about polynomials in the case when $I=J=1$.

\begin{lemma}
\label{lemPolynomialsFromOneToOne}
For fixed objects $I$ and $J$ of a locally cartesian closed category $\mathcal{E}$, there are isomorphisms of categories
$$S : \mathbf{Poly}_{\mathcal{E}}(I,J) \overset{\cong}{\longrightarrow} \mathbf{Poly}_{\mathcal{E}\slice{I \times J}}(1,1) \quad \text{and} \quad S\cart : \mathbf{Poly}\cart_{\mathcal{E}}(I,J) \overset{\cong}{\longrightarrow} \mathbf{Poly}\cart_{\mathcal{E} \slice{I \times J}}(1,1)
\eqpunct{.}$$
\end{lemma}

\begin{proof}[Proof sketch]
Given a polynomial $F : \declpoly IBAJsft$, define $S(F) = \langle s , f \rangle : B \to I \times A$ over $I \times J$ (considered as a polynomial $1 \pto 1$ in $\mathcal{E} \slice{I \times J}$) as in
\begin{center}
\begin{tikzcd}[row sep=huge, column sep=huge]
B
\arrow[rr, "{\langle s, f \rangle}"]
\arrow[dr, "{\langle s, t \circ f \rangle}"']
&&
I \times A
\arrow[dl, "{\mathrm{id}_I \times t}"]
\\
&
I \times J
&
\end{tikzcd}
\eqpunct{.}
\end{center}

Given a morphism of polynomials $\varphi : F \pRightarrow G$, as in
\begin{center}
\begin{tikzcd}[row sep=huge, column sep=huge]
&
B
\arrow[r,"f"]
\arrow[dl, bend right=20, "s"']
&
A
\arrow[d, equals]
\arrow[dr, bend left=20, "t"]
&
\\
I
&
D_{\varphi}
\pullback
\arrow[r]
\arrow[d,"\varphi_1"']
\arrow[u,"\varphi_2"]
&
A
\arrow[d,"\varphi_0"]
&
J
\\
&
D
\arrow[r,"g"']
\arrow[ul, bend left=20, "u"]
&
C
\arrow[ur, bend right=20, "v"']
&
\end{tikzcd}
\eqpunct{,}
\end{center}
define $S(\varphi) = (\mathrm{id}_I \times \varphi_0, \varphi_1, \varphi_2) : S(F) \pRightarrow S(G)$, as in the following diagram, where we consider $E$ as an object over $I \times J$ via $\langle s \circ \varphi_2, t \circ f \circ \varphi_2 \rangle : E \to I \times J$.
\begin{center}
\begin{tikzcd}[row sep=huge, column sep=huge]
B
\arrow[r,"{\langle s, f \rangle}"]
&
I \times A
\arrow[d, equals]
\\
D_{\varphi}
\pullback
\arrow[r]
\arrow[d,"\varphi_1"']
\arrow[u,"\varphi_2"]
&
I \times A
\arrow[d,"{\mathrm{id}_I \times \varphi_0}"]
\\
D
\arrow[r,"{\langle u, g \rangle}"']
&
I \times C
\end{tikzcd}
\end{center}

It is easy to see that $\mathrm{id}_I \times \varphi_0$, $\varphi_1$ and $\varphi_2$ are morphisms over $I \times J$ and that the lower square of the above diagram truly is cartesian, so that $S(\varphi)$ is a morphism in $\mathbf{Poly}_{\mathcal{E} \slice{I \times J}}(1,1)$. Verifying functoriality and invertibility of $S$ is elementary but tedious.

That $S$ restricts to an isomorphism $S\cart : \mathbf{Poly}\cart_{\mathcal{E}}(I,J) \to \mathbf{Poly}\cart_{\mathcal{E}\slice{I \times J}}(1,1)$ is immediate, since $S(\varphi)$ is cartesian if and only if $\varphi_2$ is invertible, which holds if and only if $\varphi$ is cartesian.
\end{proof}

\newpage
\section{Review of natural models}
\label{secNaturalModels}

Natural models \cite{Awodey2016} are a notion equivalent to that of categories with families \cite{Dybjer1995} which provide a natural setting for interpreting type theory (see also \cite{Fiore2012}).

\begin{definition}
\label{defNaturalModel}
A \textbf{natural model} is a category $\mathbb{C}$ together with the following data:
\begin{itemize}
\item A terminal object $\diamond$;
\item A map of presheaves $p : \dot{\mathcal{U}} \to \mathcal{U}$ over $\mathbb{C}$;
\item \textit{Representability data.} For each object $\Gamma$ of $\mathbb{C}$ and each element $A \in \mathcal{U}(\Gamma)$, an object $\Gamma.A$, a morphism $\mathsf{p}_A : \Gamma.A \to \Gamma$ in $\mathbb{C}$ and an element $\mathsf{q}_A \in \dot{\mathcal{U}}(\Gamma)$, such that for all $\Gamma$ and all $A$, the following square is a pullback:
\begin{center}
\begin{tikzcd}[column sep=huge, row sep=huge]
\mathsf{y}(\Gamma.A)
\arrow[r, "\mathsf{q}_A"]
\arrow[d, "\mathsf{y}(\mathsf{p}_A)"']
&
\dot{\mathcal{U}}
\arrow[d, "p"]
\\
\mathsf{y}(\Gamma)
\arrow[r, "A"']
&
\mathcal{U}
\end{tikzcd}
\eqpunct{.}
\end{center}
Here $\mathsf{y} : \mathbb{C} \to \widehat{\mathbb{C}} = \mathbf{Set}^{\mathbb{C}^{\mathrm{op}}}$ is the Yoneda embedding, and we have identified $A$ and $\mathsf{q}_A$ with the corresponding natural transformations by the Yoneda lemma.
\end{itemize}
\end{definition}

Definition \ref{defNaturalModel} says essentially that a natural model is a category equipped with a terminal object and a \textit{representable natural transformation} $p : \dot{\mathcal{U}} \to \mathcal{U}$, but we include in the definition the data witnessing the representability of $p$.

In what follows, we will write $\nameof{A}$ to denote the fibre $\dot{\mathcal{U}}_A$ of $p$ over $A:\mathcal{U}$ in the internal langauge of $\widehat{\mathbb{C}}$.

\begin{remark}
Regarding $\mathbb{C}$ as a category of \textit{contexts} and \textit{substitutions} of dependent type theory \cite{MartinLof1984}, we can regard $\mathcal{U}$ as a \textit{presheaf of types}, $\dot{\mathcal{U}}$ as a \textit{presheaf of terms} and $p : \dot{\mathcal{U}} \to \mathcal{U}$ as the map sending a term to its unique type. Specifically, the elements of the set $\mathcal{U}(\Gamma)$ are the types $A$ in context $\Gamma$, and the elements $a \in \dot{\mathcal{U}}(\Gamma)$ with $p_{\Gamma}(a)=A$ are the terms $a$ of type $A$ in context $\Gamma$. Representability of $p$ allows us to form the \textit{context extension} $\Gamma.A$ of a context $\Gamma$ by a type $A$ in context $\Gamma$. This is explained in detail in \cite{Awodey2016}; a treatment of signatures for dependent type theory in a natural model is forthcoming work.
\end{remark}

In \cite{Awodey2016}, necessary and sufficient conditions are given for a natural model to support a unit type, dependent sum types and dependent product types, as summarised in the following.

\begin{theorem}
\label{thmSupportingUnitSigmaPi}
Let $\mathsf{C} = (\mathbb{C},p)$ be a natural model.
\begin{enumabc}
\item $\mathsf{C}$ supports a unit type if and only if there exist morphisms $\widehat{\mathbf{1}}$ and $\widehat{\star}$ fitting into the following pullback square:
\begin{center}
\begin{tikzcd}[row sep=huge, column sep=huge]
1
\arrow[d, equals]
\arrow[r, "\widehat{\star}"]
\pullback
&
\dot{\mathcal{U}}
\arrow[d, "p"]
\\
1
\arrow[r, "\widehat{\mathbf{1}}"']
&
\mathcal{U}
\end{tikzcd}
\eqpunct{.}
\end{center}

\item $\mathsf{C}$ supports dependent sum types ($\Sigma$-types) if and only if there exist morphisms $\widehat{\Sigma}$ and $\widehat{\mathsf{pair}}$ fitting into the following pullback square:
\begin{center}
\begin{tikzcd}[row sep=huge, column sep=huge]
\sum_{(A,B) : \sum_{A:\mathcal{U}} \mathcal{U}^{[A]}} \sum_{a : \nameof{A}} \nameof{B(a)}
\arrow[d, "\pi"']
\arrow[r, "\widehat{\mathsf{pair}}"]
\pullbackc{dr}{-0.05}
&
\dot{\mathcal{U}}
\arrow[d, "p"]
\\
\sum_{A : \mathcal{U}} \mathcal{U}^{\nameof{A}}
\arrow[r, "\widehat{\Sigma}"']
&
\mathcal{U}
\end{tikzcd}
\eqpunct{.}
\end{center}
where $\pi$ is the projection morphism;

\item $\mathsf{C}$ supports dependent product types ($\Pi$-types) if and only if there exist morphisms $\widehat{\Pi}$ and $\widehat{\lambda}$ fitting into the following pullback square:
\begin{center}
\begin{tikzcd}[row sep=huge, column sep=huge]
\sum_{A : \mathcal{U}} \dot{\mathcal{U}}^{\nameof{A}}
\arrow[d, "p'"']
\arrow[r, "\widehat{\lambda}"]
\pullbackc{dr}{-0.05}
&
\dot{\mathcal{U}}
\arrow[d, "p"]
\\
\sum_{A : \mathcal{U}} \mathcal{U}^{\nameof{A}}
\arrow[r, "\widehat{\Pi}"']
&
\mathcal{U}
\end{tikzcd}
\eqpunct{.}
\end{center}
where $p' = \sum_{A:\mathcal{U}} p^{[A]}$.
\end{enumabc}
\end{theorem}

Since $p : \dot{\mathcal{U}} \to \mathcal{U}$ is a morphism in $\widehat{\mathbb{C}}$, which is a locally cartesian closed category, it can be viewed as a polynomial $1 \pto 1$, where $1 = \mathsf{y}(\diamond)$ is our choice of terminal object in $\widehat{\mathbb{C}}$. Furthermore, observe that the morphism $\pi$ in (b) is the polynomial composite $p \cdot p$, and the morphism $p'$ in (c) is $P_p(p)$, where $P_p$ is the extension of $p$. We can therefore rephrase the statement of Theorem \ref{thmSupportingUnitSigmaPi} in terms of morphisms of polynomials.

\begin{theorem}
\label{thmUnitSigmaPiPoly}
Let $\mathsf{C} = (\mathbb{C},p)$ be a natural model.
\begin{enumabc}
\item $\mathsf{C}$ supports unit types if and only if there is a cartesian morphism $\eta : i_1 \pRightarrow p$ in $\mathbf{Poly}_{\widehat{\mathbb{C}}}$;
\item $\mathsf{C}$ supports dependent sum types if and only if there is a cartesian morphism $\mu : p \cdot p \pRightarrow p$ in $\mathbf{Poly}_{\widehat{\mathbb{C}}}$;
\item $\mathsf{C}$ supports dependent product types if and only if there is a cartesian morphism $\zeta : P_p(p) \pRightarrow p$ in $\mathbf{Poly}_{\widehat{\mathbb{C}}}$;
\end{enumabc}
\end{theorem}

We originally conjectured that $(1,p,\eta,\mu)$ is, moreover, a polynomial monad in the sense of Definition \ref{defPolynomialMonad}, and that $(p,\zeta)$ is an algebra for this monad in a suitable sense, but this turned out to be false. For example, consider the monad unit laws $\mu \circ (\eta \cdot p) = \mathrm{id}_p = \mu \circ (p \cdot \eta)$---they state precisely that the following equations of pasting diagrams hold:

\begin{center}
\begin{tikzcd}[column sep=normal, row sep=normal]
\dot{\mathcal{U}}
\arrow[d, "p"']
\arrow[r, "(\eta \cdot p)_1"]
&
\sum_{A,B} \sum_{a:\nameof{A}} \nameof{B(a)}
\arrow[d, "p \cdot p"]
\arrow[r, "\mu_1"]
&
\dot{\mathcal{U}}
\arrow[d, "p" name={domequalsone}]
&
\dot{\mathcal{U}}
\arrow[d,"p"' name={codequalsone}]
\arrow[r,equals]
\arrow[from={domequalsone},to={codequalsone},"=" description,draw=none]
&
\dot{\mathcal{U}}
\arrow[d,"p" name={domequalstwo}]
&
\dot{\mathcal{U}}
\arrow[d, "p"' name={codequalstwo}]
\arrow[r, "(p \cdot \eta)_1"]
\arrow[from={domequalstwo},to={codequalstwo},"=" description,draw=none]
&
\sum_{A,B} \sum_{a:\nameof{A}} \nameof{B(a)}
\arrow[d, "p \cdot p"]
\arrow[r, "\mu_1"]
&
\dot{\mathcal{U}}
\arrow[d, "p"]
\\
\mathcal{U}
\arrow[r, "(\eta \cdot p)_0"']
&
\sum_{A:\mathcal{U}} \mathcal{U}^{\nameof{A}}
\arrow[r, "\mu_0"']
&
\mathcal{U}
&
\mathcal{U}
\arrow[r,equals]
&
\mathcal{U}
&
\mathcal{U}
\arrow[r, "(p \cdot \eta)_0"']
&
\sum_{A:\mathcal{U}} \mathcal{U}^{\nameof{A}}
\arrow[r, "\mu_0"']
&
\mathcal{U}
\end{tikzcd}
\eqpunct{.}
\end{center}

However, these equations do not hold strictly. Indeed, in the internal language of $\widehat{\mathbb{C}}$, we have
$$(\mu \circ (\eta \cdot p))_0(A) = \sum_{x : A} \mathbf{1} = A \times \mathbf{1} \quad \text{and} \quad (\mu \circ (p \cdot \eta))_0(A) = \sum_{x : \mathbf{1}} A = \mathbf{1} \times A
\eqpunct{.}$$
But in type theory, the types $A \times \mathbf{1}$, $A$ and $\mathbf{1} \times A$ are not equal, although there are canonical isomorphisms between them. We therefore cannot, in general, expect the monad laws to hold strictly. However, it is still reasonable to expect this structure to satisfy the laws of a \textit{pseudomonad}. As such, we require a suitable notion of equivalence between morphisms of polynomials---however, this is not currently available to us, since $\mathbf{Poly}_{\widehat{\mathbb{C}}}$ is merely a bicategory.

In Section \ref{secTricategory}, we will equip $\mathbf{Poly}_{\mathcal{E}}$ with $3$-cells, endowing it with the structure of a $\mathbf{2Cat}$-enriched bicategory, which is a kind of tricategory with strict composition in dimension $2$. This affords us the ability to show that we have a \textit{polynomial pseudomonad} and a \textit{polynomial pseudoalgebra}---proving this is then the content of Section \ref{secMonadAlgebra}.

\newpage
\section{Polynomial pseudomonads and pseudoalgebras}
\label{secTricategory}

Much as monads naturally live in bicategories, pseudomonads naturally live in \textit{tricategories}. To define the notion of a polynomial pseudomonad, we therefore need to endow the bicategory $\mathbf{Poly}\cart_{\mathcal{E}}$ with $3$-cells turning it into a tricategory.

\subsection{A tricategory of polynomials}

In general, tricategories are fiddly, with lots of coherence data to worry about \cite{Gurski2013}---fortunately for us, our situation is simplified by the fact that composition of 2-cells of polynomials is strict, so that the $3$-cells turn the hom categories $\mathbf{Poly}_{\mathcal{E}}(I,J)$ into $2$-categories, rather than bicategories. The emerging structure is that of a \textit{$\mathbf{2Cat}$-enriched bicategory}.

\begin{definition}
\label{def2CatEnrichedBicategory}
A \textbf{$\mathbf{2Cat}$-enriched bicategory} $\mathfrak{B}$ consists of:
\begin{itemize}
\item A set $\mathfrak{B}_0$, whose elements we call the \textbf{0-cells} of $\mathfrak{B}$;
\item For all 0-cells $I,J$, a 2-category $\mathfrak{B}(I,J)$, whose 0-cells, 1-cells and 2-cells we call the \textbf{1-cells}, \textbf{2-cells} and \textbf{3-cells} of $\mathfrak{B}$, respectively;
\item For all 0-cells $I,J,K$, a 2-functor $\circ_{I,J,K} : \mathfrak{B}(J,K) \times \mathfrak{B}(I,J) \to \mathfrak{B}(I,K)$, which we call the \textbf{composition} 2-functor;
\item For all 0-cells $I$, a 2-functor $\iota_I : \mathbf{1} \to \mathfrak{B}(I,I)$, which we call the \textbf{identity} 2-functor, where $\mathbf{1}$ is the terminal 2-category;
\item For all 0-cells $I,J,K,L$, a 2-natural isomorphism
\begin{center}
\begin{tikzcd}[column sep=huge, row sep=huge]
\mathfrak{B}(K,L) \times \mathfrak{B}(J,K) \times \mathfrak{B}(I,J)
\arrow[r, "\circ_{J,K,L} \times \mathrm{id}"]
\arrow[d, "\mathrm{id} \times \circ_{I,J,K}"']
&
\mathfrak{B}(J,L) \times \mathfrak{B}(I,J)
\arrow[d, "\circ_{I,J,L}"]
\arrow[dl, draw=none, pos=0.5, "\phantom{\alpha_{I,J,K,L}}\ \twocell{225}\ \alpha_{I,J,K,L}" description]
\\
\mathfrak{B}(K,L) \times \mathfrak{B}(I,K)
\arrow[r, "\circ_{I,K,L}"']
&
\mathfrak{B}(I,L)
\end{tikzcd}
\eqpunct{,}
\end{center}
called the \textbf{associator};

\item For all 0-cells $I,J$, 2-natural isomorphisms

\begin{center}
\begin{tikzcd}[row sep=huge, column sep={20pt}]
\mathfrak{B}(I,J) \times \mathbf{1}
\arrow[r, "\mathrm{id} \times \iota_I"]
\arrow[dr, "\cong"']
&
\mathfrak{B}(I,J) \times \mathfrak{B}(I,I)
\arrow[d, "\circ_{I,I,J}"]
\arrow[dl, draw=none, pos=0.2, "\lambda_{I,J} \twocell{225} \phantom{\lambda_{I,J}}" description]
\\
~
&
\mathfrak{B}(I,J)
&
\end{tikzcd}
\begin{tikzcd}[row sep=huge, column sep={20pt}]
\mathfrak{B}(J,J) \times \mathfrak{B}(I,J)
\arrow[d, "\circ_{I,J,J}"']
\arrow[dr, draw=none, pos=0.2, "\rho_{I,J} \twocell{315} \phantom{\rho_{I,J}}" description]
&
\mathbf{1} \times \mathfrak{B}(I,J)
\arrow[l, "\iota_J \times \mathrm{id}"']
\arrow[dl, "\cong"]
\\
\mathfrak{B}(I,J)
&
~
\end{tikzcd}
\eqpunct{,}
\end{center}
called the \textbf{left unitor} and \textbf{right unitor}, respectively.
\end{itemize}
such that for all compatible 1-cells $I \xrightarrow{f} J \xrightarrow{g} K \xrightarrow{h} L \xrightarrow{k} M$, the following diagrams commute:
\begin{center}
\begin{tikzcd}[row sep=huge, column sep={0pt}]
((k \circ h) \circ g) \circ f
\arrow[rr, Rightarrow, "\alpha_{I,J,K,M}"]
\arrow[dr, Rightarrow, "\alpha_{J,K,L,M} \circ f"']
&&
(k \circ h) \circ (g \circ f)
\arrow[rr, Rightarrow, "\alpha_{I,K,L,M}"]
&&
k \circ (h \circ (g \circ f))
\\
&
(k \circ (h \circ g)) \circ f
\arrow[rr, Rightarrow, "\alpha_{I,J,L,M}"']
&&
k \circ ((h \circ g) \circ f)
\arrow[ur, Rightarrow, "k \circ \alpha_{I,J,K,L}"']
&
\end{tikzcd}
\eqpunct{,}
\end{center}

\begin{center}
\begin{tikzcd}[row sep=huge, column sep=huge]
(g \circ \iota_J) \circ f
\arrow[rr, Rightarrow, "\alpha_{I,J,J,K}"]
\arrow[dr, Rightarrow, "\lambda_{J,K} \circ f"']
&&
g \circ (\iota_J \circ f)
\arrow[dl, Rightarrow, "g \circ \rho_{I,J}"]
\\
&
g \circ f
&
\end{tikzcd}
\eqpunct{.}
\end{center}
\end{definition}

\begin{remark}
Every 3-category is trivially a $\mathbf{2Cat}$-enriched bicategory, and every $\mathbf{2Cat}$-enriched bicategory is a tricategory. Every $\mathbf{2Cat}$-enriched bicategory has an underlying bicategory, obtained by forgetting the 3-cells, and every bicategory can be equipped with the structure of a $\mathbf{2Cat}$-enriched bicategory by taking only identities as 3-cells. An equivalent viewpoint is that $\mathbf{2Cat}$-enriched bicategories are tricategories, whose hom-bicategories are $2$-categories and whose coherence isomorphisms in the top dimension are identities.
\end{remark}

Connections between polynomials and $\mathbf{2Cat}$-enriched bicategories have been studied in different but related settings by Tamara von Glehn \cite{vonGlehn2015} and by Mark Weber \cite{Weber2015} (the latter referring to them as `2-bicategories').

In order to motivate our definition of 3-cells, we make an observation relating polynomials with internal categories. First, we recall (e.g.\ \cite{Jacobs2001}) the definition of the \textit{internal full subcategory} associated with a morphism in a locally cartesian closed category.

\begin{definition}
\label{defInternalFullSubcategory}
Let $f : B \to A$ be a morphism in a locally cartesian closed category $\mathcal{E}$. The \textbf{internal full subcategory} of $\mathcal{E}$ associated with $f$ is the internal category $\mathbb{A}_f$ whose object of objects is $A$ and whose object of morphisms is $\sum_{a,a' \in A} B_{a'}^{B_a}$, with the projections $\partial_0,\partial_1$ to $A$ giving the domain and codomain morphisms, and with identity and composition morphisms defined in the obvious way.

Explicitly, the morphism $\partial = \langle \partial_0, \partial_1 \rangle : (\mathbb{A}_f)_1 \to A \times A$ is defined to be the exponential object $f_2^{f_1}$ in $\mathcal{E}\slice{A \times A}$, where $f_1$ and $f_2$ are the given by pulling back $f$ along the projections $\pi_1,\pi_2 : A \times A \rightrightarrows A$.
\end{definition}

\begin{theorem}
\label{thmPolynomialsYieldInternalCategories}
Fix objects $I$ and $J$ in a locally cartesian closed category $\mathcal{E}$. There is a functor
$$\mathbb{A}_{(-)} : \mathbf{Poly}\cart_{\mathcal{E}}(I,J) \to \mathbf{Cat}(\mathcal{E} \slice{I \times J})
\eqpunct{.}$$
Moreover, every functor of the form $\mathbb{A}_{\varphi}$ is full and faithful.
\end{theorem}

\begin{proof}
We assume $I=J=1$, letting Lemma \ref{lemPolynomialsFromOneToOne} take care of the general case.

Given a morphism $f : B \to A$ of $\mathcal{E}$, let $\mathbb{A}_f$ be the internal full subcategory of $\mathcal{E}$ associated with $f$ (as in Definition \ref{defInternalFullSubcategory}).

Given a cartesian morphism of polynomials $\varphi : f \pRightarrow g$, represented by the following pullback square:
\begin{center}
\begin{tikzcd}[column sep=huge, row sep=huge]
B
\arrow[r, "\varphi_1"]
\arrow[d, "f"']
\pullback
&
D
\arrow[d, "g"]
\\
A
\arrow[r, "\varphi_0"']
&
C
\end{tikzcd}
\eqpunct{,}
\end{center}
let $\mathbb{A}_{\varphi} : \mathbb{A}_f \to \mathbb{A}_g$ be the internal functor defined as follows. The action of $\mathbb{A}_{\varphi}$ on objects is given by $\varphi_0 : A \to C$. Since $f \cong \Delta_{\varphi}(g)$ and pullbacks preserve exponentials in locally cartesian closed categories, it follows that $f_2^{f_1} \cong \Delta_{\varphi_0 \times \varphi_0}(g_2^{g_1})$. This determines a canonical morphism $(\mathbb{A}_{\varphi})_1 : (\mathbb{A}_f)_1 \to (\mathbb{A}_g)_1$, as in the following pullback square:
\begin{center}
\begin{tikzcd}[row sep=huge, column sep=huge]
\sum_{a,a' \in A} B_{a'}^{B_a}
\arrow[r, dashed, "(\mathbb{A}_{\varphi})_1"]
\arrow[d, "f_2^{f_1}"']
\pullbackc{dr}{-0.05}
&
\sum_{c,c' \in C} D_{c'}^{D_c}
\arrow[d, "g_2^{g_1}"]
\\
A \times A
\arrow[r, "\varphi_0 \times \varphi_0"']
&
C \times C
\end{tikzcd}
\eqpunct{.}
\end{center}
It is easy to verify that $\mathbb{A}_{\varphi}$ is an internal functor and that $\mathbb{A}_{\psi \circ \varphi} = \mathbb{A}_{\psi} \circ \mathbb{A}_{\varphi}$ for all composable pairs of cartesian morphisms $\varphi,\psi$. The fact that $\mathbb{A}_{\varphi}$ is full and faithful is expressed precisely by the fact that the square defining $(\mathbb{A}_{\varphi})_1$ is cartesian.
\end{proof}

\begin{remark}
Theorem \ref{thmPolynomialsYieldInternalCategories} yields a 1-functor between 1-categories. However, $\mathbf{Cat}(\mathcal{E} \slice{I \times J})$ has the structure of a 2-category, so it is therefore reasonable to expect that when we equip $\mathbf{Poly}_{\mathcal{E}}$ with 3-cells, the functor $\mathbb{A}_{(-)}$ should extend to a 2-functor. In particular, any 3-cell between cartesian morphisms of polynomials should induce an internal natural transformation between the induced internal functors. However, since the association of internal functors to morphisms of polynomials works only for \textit{cartesian} morphisms of polynomials, we cannot simply take internal natural transformations as the 3-cells of $\mathbf{Poly}_{\mathcal{E}}$. Lemmas \ref{lemHalfInternalNT} and \ref{lemCartesianAdjustmentsAreInternalNT} provide a correspondence between internal natural transformations $\mathbb{A}_{\varphi} \Rightarrow \mathbb{A}_{\psi}$ and particular morphisms of $\mathcal{E}$ in a way that generalises to the case when $\varphi$ and $\psi$ are not required to be cartesian.
\end{remark}

\begin{lemma}
\label{lemHalfInternalNT}
Let $f : B \to A$ and $g : D \to C$ be polynomials in a locally cartesian closed category $\mathcal{E}$ and let $\varphi, \psi : f \pRightarrow g$ be cartesian morphisms of polynomials. There is a bijection between the set of morphisms $\alpha : \Delta_{\varphi_0} D \to \Delta_{\psi_0} D$ in $\mathcal{E} \slice{A}$ and the set of morphisms $\widehat{\alpha} : A \to (\mathbb{A}_g)_1$ in $\mathcal{E} \slice{C \times C}$, as indicated by dashed arrows in the following diagrams:
\begin{center}
\begin{tikzcd}[row sep=small, column sep=small]
\Delta_{\varphi} D
\arrow[rr, dashed, blue, "\alpha"]
\arrow[d, "\varphi_2"', "{\scriptsize \cong}"]
&&
\Delta_{\psi_0} D
\arrow[d, "\psi_2", "{\scriptsize \cong}"']
\\
B
\arrow[dr, "f"', bend right]
&&
B
\arrow[dl, "f", bend left]
\\
&
A
&
\end{tikzcd}
\hspace{30pt}
\begin{tikzcd}[row sep=huge, column sep=huge]
&
(\mathbb{A}_g)_1
\arrow[d, "\partial"]
\\
A
\arrow[ur, dashed, blue, "\widehat{\alpha}"]
\arrow[r, "{\langle \varphi_0, \psi_0 \rangle}"']
&
C \times C
\end{tikzcd}
\eqpunct{,}
\end{center}
where $\varphi_2, \psi_2$ are the canonical isomorphisms induced by the universal property of pullbacks, as in Remark \ref{rmkCartesianMorphismIsPullbackSquare}.
\end{lemma}

\begin{proof}
Given $\alpha : \Delta_{\varphi_0} D \to \Delta_{\psi_0} D$ in $\mathcal{E} \slice{A}$, the exponential transpose of $\alpha$ in $\mathcal{E} \slice{A}$ is, as a morphism in $\mathcal{E}$, a section $\overline{\alpha} : A \to H$ of the projection $H \to A$, where $H = \sum_{a \in A} D_{\psi_0(a)}^{D_{\varphi_0(a)}}$. This projection is precisely the pullback of $(\mathbb{A}_g)_1 \to C \times C$ along $\langle \varphi_0, \psi_0 \rangle$, as  illustrated in the following diagram:
\begin{center}
\begin{tikzcd}[row sep=huge, column sep=huge]
H
\arrow[d]
\arrow[r]
\pullback
&
(\mathbb{A}_g)_1
\arrow[d, "\partial"]
\\
A
\arrow[r, "{\langle \varphi_0, \psi_0 \rangle}"']
\arrow[u, bend left, dashed, blue, "\overline{\alpha}"]
&
C \times C
\end{tikzcd}
\eqpunct{.}
\end{center}
But sections of the pullback correspond with diagonal fillers $\widehat{\alpha} : A \to (\mathbb{A}_g)_1$ of the pullback square. This is as required, since such a filler making the lower triangle commute makes the upper triangle commute automatically. This concludes the proof of (a).
\end{proof}

\begin{lemma}
\label{lemCartesianAdjustmentsAreInternalNT}
Let $f : B \to A$ and $g : D \to C$ be polynomials in a locally cartesian closed category $\mathcal{E}$, let $\varphi,\psi : f \pRightarrow g$ be cartesian morphisms of polynomials, and let $\alpha,\widehat{\alpha}$ be as in Lemma \ref{lemHalfInternalNT}. The following are equivalent:
\begin{enumrom}
\item $\widehat{\alpha}$ is an internal natural transformation $\mathbb{A}_{\varphi} \Rightarrow \mathbb{A}_{\psi}$;
\item In the internal language of $\mathcal{E}$, we have $\seqbn{\mathbb{A}_{\psi}(k) \circ \alpha_a = \alpha_{a'} \circ \mathbb{A}_{\varphi}(k)}{a,a' \in A,\ k \in B_{a'}^{B_a}}$;
\item In the internal language of $\mathcal{E}$, we have $\seqbn{\gamma_{a'} \circ k = k \circ \gamma_a}{a,a' \in A,\ k \in B_{a'}^{B_a}}$, where $\gamma = \psi_2 \circ \alpha \circ \varphi_2^{-1} : B \to B$;
\item $\alpha$ is a morphism in $\mathcal{E} \slice{B}$, i.e.\ $\psi_2 \circ \alpha = \varphi_2$.
\end{enumrom}
\end{lemma}

\begin{proof}
We prove (i)$\Leftrightarrow$(ii)$\Leftrightarrow$(iii)$\Leftrightarrow$(iv).

\begin{itemize}
\item[(i)$\Leftrightarrow$(ii)] In light of Lemma \ref{lemHalfInternalNT}, this is just a translation into the internal language of $\mathcal{E}$ of the definition of an internal natural transformation.

\item[(ii)$\Leftrightarrow$(iii)] Consider the following `internal' diagram, parametrised by $a,a' \in A$ and $k \in B_{a'}^{B_a}$:
\begin{center}
\begin{tikzcd}[row sep=huge, column sep=huge]
B_a
\arrow[d, "k"']
\arrow[r, "(\varphi_2)_a^{-1}"]
&
D_{\varphi(a)}
\arrow[d, "\mathbb{A}_{\varphi}(k)" description]
\arrow[r, "\alpha_a"]
&
D_{\psi(a)}
\arrow[d, "\mathbb{A}_{\psi}(k)" description]
\arrow[r, "(\psi_2)_a"]
&
B_a
\arrow[d, "k"]
\\
B_{a'}
\arrow[r, "(\varphi_2)_{a'}^{-1}"']
&
D_{\varphi(a')}
\arrow[r, "\alpha_{a'}"']
&
D_{\psi(a')}
\arrow[r, "(\psi_2)_{a'}"']
&
B_{a'}
\end{tikzcd}
\eqpunct{.}
\end{center}
The left- and right-hand squares commute by functoriality of $\mathbb{A}_{\varphi}$ and $\mathbb{A}_{\psi}$.
The centre square commutes if and only if (ii) holds, and the outer square commutes if and only if (iii) holds. But the centre square commutes if and only if the outer square commutes.

\item[(iii)$\Leftrightarrow$(iv)] Let $a \in A$ and $b \in B_a$, and let $k \in B_a^{B_a}$ be the constant (internal) function with value $b$. If (iii) holds, then
$$\gamma_a(b) = \gamma_a(k(b)) = k(\gamma_{a}(b)) = b$$
so that $\seqbn{\gamma_a = \mathrm{id}_{B_a}}{a \in A}$ holds. But this says precisely that $\gamma=\mathrm{id}_B$, and hence $\psi_2 \circ \alpha = \varphi_2$. The converse (iv)$\Rightarrow$(iii) is immediate.
\end{itemize}
\end{proof}

\begin{definition}
\label{defAdjustment}
Let $F : \declpoly IBAJsft$ and $G : \declpoly IDCJugv$ be polynomials and let $\varphi,\psi : F \pRightarrow G$ be morphisms of polynomials, as in:
\begin{center}
\begin{tikzcd}[row sep=normal, column sep=normal]
&
B
\arrow[r,"f"]
\arrow[dl, bend right=20, "s"']
&
A
\arrow[d, equals]
\arrow[dr, bend left=20, "t"]
&
& 
&
B
\arrow[r,"f"]
\arrow[dl, bend right=20, "s"']
&
A
\arrow[d, equals]
\arrow[dr, bend left=20, "t"]
&
\\
I
&
D_{\varphi}
\pullback
\arrow[r]
\arrow[d,"\varphi_1"']
\arrow[u,"\varphi_2"]
&
A
\arrow[d,"\varphi_0"]
&
J
& 
I
&
D_{\psi}
\pullback
\arrow[r]
\arrow[d,"\psi_1"']
\arrow[u,"\psi_2"]
&
A
\arrow[d,"\psi_0"]
&
J
\\
&
D
\arrow[r,"g"']
\arrow[ul, bend left=20, "u"]
&
C
\arrow[ur, bend right=20, "v"']
&
& 
&
D
\arrow[r,"g"']
\arrow[ul, bend left=20, "u"]
&
C
\arrow[ur, bend right=20, "v"']
&
\end{tikzcd}
\eqpunct{.}
\end{center}
An \textbf{adjustment} $\alpha$ from $\varphi$ to $\psi$, denoted $\alpha : \varphi \pRrightarrow \psi$, is a morphism $\alpha : D_{\varphi} \to D_{\psi}$ over $B$:
\begin{center}
\begin{tikzcd}[row sep=huge, column sep=huge]
D_{\varphi}
\arrow[rr,"\alpha"]
\arrow[dr,"\varphi_2"']
&&
D_{\psi}
\arrow[dl,"\psi_2"]
\\
&
B
&
\end{tikzcd}
\eqpunct{.}
\end{center}
\end{definition}

\begin{remark}
Lemma \ref{lemCartesianAdjustmentsAreInternalNT} tells us that, when $\varphi$ and $\psi$ are cartesian, adjustments $\alpha : \varphi \pRrightarrow \psi$ can equivalently be described as internal natural transformations $\widehat{\alpha} : \varphi \Rightarrow \psi$.
\end{remark}

\begin{conjecture}
\label{conjPolyETricategory}
There is a $\mathbf{2Cat}$-enriched bicategory $\mathfrak{Poly}_{\mathcal{E}}$, whose underlying bicategory is $\mathbf{Poly}_{\mathcal{E}}$ and whose 3-cells are adjustments.
\end{conjecture}

Unfortunately, the details required to fully prove Conjecture \ref{conjPolyETricategory} turned out to be somewhat cumbersome and, since its full force is not required for our main results, we have left the task of verifying these details for future work. Our progress so far is outlined in Lemma \ref{lemPolyEHom2Categories} and Remark \ref{rmkCompleteProofOfConjecture}, and we prove the analogous result with attention restricted to \textit{cartesian} morphisms of polynomials in Theorem \ref{thmPolyECartTrivial}.

\begin{lemma}
\label{lemPolyEHom2Categories}
Let $I$ and $J$ be objects in a locally cartesian closed category $\mathcal{E}$. There is a 2-category $\mathfrak{Poly}_{\mathcal{E}}(I,J)$ whose underlying category is $\mathbf{Poly}_{\mathcal{E}}(I,J)$ and whose 2-cells are adjustments.
\end{lemma}

\begin{proof}
Given polynomials $F,G : I \pto J$, the category $\mathfrak{Poly}_{\mathcal{E}}(I,J)(F,G)$ has morphisms of polynomials $F \pRightarrow G$ as its objects and adjustments as its morphisms, with identity and composition inherited from $\mathcal{E}\slice{B}$.

Given a polynomial $F : \declpoly IBAJsft$, we have an evident functor $\mathbf{1} \to \mathfrak{Poly}_{\mathcal{E}}(I,J)(F,F)$ picking out the identity morphism $F \pRightarrow F$ and the identity adjustment on this morphism.

Let $F,G,H : I \pto J$ be polynomials. The composition functor
$$c : \mathbf{Poly}_{\mathcal{E}}(I,J)(G,H) \times \mathbf{Poly}_{\mathcal{E}}(I,J)(F,G) \to \mathbf{Poly}_{\mathcal{E}}(I,J)(F,H)$$
is defined as follows. The composite $c(\psi,\varphi)$ of $\varphi : F \pRightarrow G$ and $\psi : G \pRightarrow H$ is defined using a pullback construction, as defined in \cite[3.9]{GambinoKock2013}---in particular, the morphism $(\psi \circ \varphi)_2 : D_{\psi \circ \varphi} \to B$ is induced by the universal property of pullbacks. This yields, for each pair of adjustments $\alpha : \varphi \pRrightarrow \varphi'$ and $\beta : \psi \pRrightarrow \psi'$, a unique morphism $D_{\psi \circ \varphi} \to D_{\psi' \circ \varphi'}$ in $\mathcal{E}$ induced by the universal property of pullbacks, which is an adjustment since it makes the required triangle in $\mathcal{E}\slice{B}$ commute. We take this morphism to be $c(\beta,\alpha)$. Functoriality of $c$ is then immediate from the universal property of pullbacks.

It can be easily verified that this data satisfies the required identity and associativity axioms. Thus we have a 2-category.
\end{proof}

\begin{remark}
\label{rmkCompleteProofOfConjecture}
In order to prove Conjecture \ref{conjPolyETricategory} in its entirety, it remains to define the coherence 2-natural isomorphisms $\alpha,\lambda,\rho$, as described in Definition \ref{def2CatEnrichedBicategory}, and verify that the required diagrams commute.

To give the reader an idea of the flavour of this task, we present some progress towards defining the associator 2-natural transformation $\alpha$. For each quadruple of objects $I,J,K,L$ of $\mathcal{E}$, this must assign to each triple of polynomials $I \overset{F}{\pto} J \overset{G}{\pto} K \overset{H}{\pto} L$ a morphism of polynomials $\alpha_{F,G,H} : (H \cdot G) \cdot F \pRightarrow H \cdot (G \cdot F)$ and, to each triple of morphisms of polynomials
$$\varphi : F \pRightarrow F', \quad \chi : G \pRightarrow G', \quad \psi : H \pRightarrow H'
\eqpunct{,}$$
an adjustment
$$\alpha_{\varphi,\chi,\psi} : \psi \cdot (\chi \cdot \varphi) \circ \alpha_{F,G,H} \pRrightarrow \alpha_{F',G',H'} \circ (\psi \cdot \chi) \cdot \varphi : (F \cdot G) \cdot H \pRightarrow F' \cdot (G' \cdot H')
\eqpunct{,}$$
which satisfy naturality laws and behave well with respect to composition and identity.

Restricting to the case $I=J=K=L=1$, let $f : B \to A$, $g : D \to C$ and $h : F \to E$ be morphisms of $\mathcal{E}$, considered as polynomials $1 \pto 1$ as usual. We will construct an invertible (and hence cartesian) morphism of polynomials $\alpha_{f,g,h} : (h \cdot g) \cdot f \pRightarrow h \cdot (g \cdot f)$. Such a morphism must fit into the following pullback square:
\begin{center}
\begin{tikzcd}[row sep=huge, column sep=huge]
\sum_{e,n,q} \sum_{f \in F_e} \sum_{d \in D_{n(f)}} B_{q(f,d)}
\arrow[d, "(h \cdot g) \cdot f"']
\arrow[r, "{(\alpha_{f,g,h})_1}"]
\pullbackc{dr}{-0.1}
&
\sum_{e,p} \sum_{f \in F_e} \sum_{d \in D_{c_f}} B_{m_f(d)}
\arrow[d, "h \cdot (g \cdot f)"]
\\
\sum_{e \in E} \sum_{n \in C^{F_e}} \prod_{f \in F_e} \prod_{d \in D_{n(f)}} A
\arrow[r, "{(\alpha_{f,g,h})_0}"']
&
\sum_{e \in E} \prod_{f \in F_e} \sum_{c \in C} \prod_{d \in D_c} A
\end{tikzcd}
\eqpunct{.}
\end{center}
In the above, we have overloaded the letter $f$, which is ambiguous between the morphism $f : B \to A$ of $\mathcal{E}$ and an internal `element' $f \in F_e$; and we have written $p(f)=(c_f,m_f)$ for $p \in \prod_{f \in F_e} \sum_{c \in C} \prod_{d \in D_c} A$ and $f \in F_e$.

The isomorphism $(\alpha_{f,g,h})_0$ is given by applying the type theoretic axiom of choice to exchange the middle $\Sigma\Pi$. Specifically, we have
$$(\alpha_{f,g,h})_0(e,n,q) = (e, \lambda f. \langle n(f), q(f) \rangle)
\eqpunct{.}$$
The isomorphism $(\alpha_{f,g,h})_1$ acts trivially; that is, we have
$$(\alpha_{f,g,h})_1(e,n,q,f,d,b) = ((\alpha_{f,g,h})_0(e,n,q),f,d,b)
\eqpunct{.}$$

We suspect that the definition of $\alpha_{\varphi,\chi,\psi}$ will also be an instance of the type theoretic axiom of choice. From this, it will be an exercise in symbolic manipulations to check that the `Mac Lane pentagon' will commute.
\end{remark}

The situation in which we restrict our attention to cartesian morphisms of polynomials is greatly simplified by the following lemma, allowing us to prove Conjecture \ref{conjPolyETricategory} for this case in Theorem \ref{thmPolyECartTrivial}.

\begin{lemma}
\label{lemTriviality}
Let $\varphi$ and $\psi$ be morphisms of polynomials. If $\psi$ is cartesian then there is a unique adjustment from $\varphi$ to $\psi$.
\end{lemma}

\begin{proof}
When $\psi$ is cartesian, the morphism $\psi_2$ is invertible, so that $\alpha = \psi_2^{-1} \circ \varphi_2$ is the only morphism making the required triangle commute.
\end{proof}

From Theorem \ref{thmPolyEBicategory}(d) and Lemma \ref{lemTriviality}, we immediately obtain the following theorem.

\begin{theorem}
\label{thmPolyECartTrivial}
There is a $\mathbf{2Cat}$-enriched bicategory $\mathfrak{Poly}\cart_{\mathcal{E}}$ (which, modulo Conjecture \ref{conjPolyETricategory}, is a sub-$\mathbf{2Cat}$-enriched bicategory of $\mathfrak{Poly}_{\mathcal{E}}$), whose underlying bicategory is $\mathbf{Poly}\cart_{\mathcal{E}}$ and whose hom 2-categories $\mathfrak{Poly}\cart_{\mathcal{E}}(I,J)$ are locally codiscrete for all objects $I,J$ of $\mathcal{E}$.
\end{theorem}

\begin{proof}
The description of the $\mathbf{2Cat}$-enriched bicategory data is given in the work towards a proof of Conjecture \ref{conjPolyETricategory}. The coherence data is uniquely defined and satisfies the required equations by Lemma \ref{lemTriviality}.
\end{proof}

Before moving on, we extend Lemma \ref{lemPolynomialsFromOneToOne} to our tricategorical setting.

\begin{lemma}
\label{lemPolynomialsFromOneToOneTwoCategorical}
For fixed objects $I$ and $J$ of a locally cartesian closed category $\mathcal{E}$, there are isomorphisms of 2-categories
$$S : \mathfrak{Poly}_{\mathcal{E}}(I,J) \overset{\cong}{\longrightarrow} \mathfrak{Poly}_{\mathcal{E}\slice{I \times J}}(1,1) \quad \text{and} \quad S\cart : \mathfrak{Poly}\cart_{\mathcal{E}}(I,J) \overset{\cong}{\longrightarrow} \mathfrak{Poly}\cart_{\mathcal{E} \slice{I \times J}}(1,1)
\eqpunct{.}$$
\end{lemma}

\begin{proof}
Let $F : \declpoly IBAJsft$ and $G : \declpoly IDCJugv$ be polynomials $I \pto J$, and let $\varphi,\psi$ be morphisms of polynomials $F \pRightarrow G$. An adjustment $\alpha : \varphi \pRrightarrow \psi$ is simply a morphism $\alpha : \varphi_2 \to \psi_2$ in $\mathcal{E}\slice{B}$. Since $S(\varphi)_2 = \varphi_2$ and $S(\psi)_2 = \psi_2$, an adjustment $S(\varphi) \pRrightarrow S(\psi)$ is a morphism $\varphi_2 \to \psi_2$ in $(\mathcal{E} \slice{I \times J}) \slice{\langle s, t \circ f \rangle} \cong \mathcal{E} \slice{B}$. So we can take $S$ to be the identity on adjustments. This trivially extends the functors $S$ and $S\cart$ of Lemma \ref{lemPolynomialsFromOneToOne} to 2-functors.
\end{proof}

\begin{theorem}
\label{thmPolynomialsYieldInternalCategories2Functor}
Fix objects $I$ and $J$ in a locally cartesian closed category $\mathcal{E}$. There is a locally full and faithful 2-functor
$$\mathbb{A}_{(-)} : \mathfrak{Poly}\cart_{\mathcal{E}}(I,J) \to \mathbf{Cat}(\mathcal{E} \slice{I \times J})
\eqpunct{,}$$
whose underlying 1-functor is as in Theorem \ref{thmPolynomialsYieldInternalCategories}.
\end{theorem}

\begin{proof}
Let $\varphi,\psi : F \pRightarrow G$ be cartesian morphisms of polynomials $I \pto J$. We proved in Lemma \ref{lemCartesianAdjustmentsAreInternalNT} that adjustments $\alpha : \varphi \pRrightarrow \psi$ correspond bijectively with internal natural transformations $\widehat{\alpha} : \mathbb{A}_{\varphi} \Rightarrow \mathbb{A}_{\psi}$. Moreover, by Lemma \ref{lemTriviality}, there is a unique internal natural transformation $\mathbb{A}_{\varphi} \Rightarrow \mathbb{A}_{\psi}$. As such, defining $\mathbb{A}_{\alpha} = \widehat{\alpha}$ for all adjustments $\alpha$, we automatically obtain a 2-functor, which is locally full and faithful since the hom-sets $\mathfrak{Poly}\cart_{\mathcal{E}}(I,J)(F,G)(\varphi,\psi)$ and $\mathbf{Cat}(\mathcal{E}\slice{I \times J})(\mathbb{A}_F,\mathbb{A}_G)(\mathbb{A}_{\varphi},\mathbb{A}_{\psi})$ are both singletons.
\end{proof}

\subsection{Polynomial pseudomonads}

We are now ready to define the notion of a polynomial pseudomonad. First, we recall the definition of a pseudomonad in a $\mathbf{2Cat}$-enriched bicateogry (in fact, the definition works just fine in an arbitrary tricategory).

\begin{definition}
\label{defPseudomonadIn2CatEnrichedBicategory}
Let $\mathfrak{B}$ be a $\mathbf{2Cat}$-enriched bicategory. A \textbf{pseudomonad} $\mathbb{T}$ in $\mathfrak{B}$ consists of:
\begin{itemize}
\item A 0-cell $I$ of $\mathfrak{B}$;
\item A 1-cell $t : I \to I$;
\item 2-cells $\eta : \mathrm{id}_I \Rightarrow t$ and $\mu : t \cdot t \Rightarrow t$, called the \textbf{unit} and \textbf{multiplication} of the pseudomonad, respectively;
\item Invertible 3-cells $\alpha, \lambda, \rho$, called the \textbf{associator}, \textbf{left unitor} and \textbf{right unitor} of the pseudomonad, respectively, as in
\begin{center}
\begin{tikzcd}[row sep=huge, column sep=huge]
t \cdot t \cdot t
\arrow[r, Rightarrow, "t \cdot \mu"]
\arrow[d, Rightarrow, "\mu \cdot t"']
&
t \cdot t
\arrow[d, Rightarrow, "\mu"]
\arrow[dl, draw=none, pos=0.5, "\alpha \threecell{225}" description]
&
t
\arrow[r, Rightarrow, "t \cdot \eta"]
\arrow[dr, Rightarrow, "\mathrm{id}_t"']
&
t \cdot t
\arrow[d, Rightarrow, "\mu" description]
\arrow[dl, draw=none, pos=0.25, "\lambda \threecell{225}" description]
\arrow[dr, draw=none, pos=0.25, "\rho \threecell{315}" description]
&
t
\arrow[l, Rightarrow, "\eta \cdot t"']
\arrow[dl, Rightarrow, "\mathrm{id}_t"]
\\
t \cdot t
\arrow[r, Rightarrow, "\mu"']
&
t
&
~
&
t
&
~
\end{tikzcd}
\eqpunct{.}
\end{center}
\end{itemize}
such that the following equations of pasting diagrams hold:

\begin{center}
\begin{tikzcd}[row sep=huge, column sep={40pt}]
t \cdot t \cdot t \cdot t
\arrow[r, Rightarrow, "t \cdot t \cdot \mu"]
\arrow[d, Rightarrow, "\mu \cdot t \cdot t"']
\arrow[dr, Rightarrow, "t \cdot \mu \cdot t" description]
&
t \cdot t \cdot t
\arrow[dr, Rightarrow, "t \cdot \mu"]
\arrow[d, draw=none, pos=0.5, "t \cdot \alpha\ \threecell{270} \phantom{T\alpha}" description]
&
&
t \cdot t \cdot t \cdot t
\arrow[r, Rightarrow, "t \cdot t \cdot \mu"]
\arrow[d, Rightarrow, "\mu \cdot t \cdot t"']
\arrow[dr, draw=none, pos=0.5, "\cong" description]
&
t \cdot t \cdot t
\arrow[dr, Rightarrow, "t \cdot \mu"]
\arrow[d, Rightarrow, "\mu \cdot t" description]
&
\\
t \cdot t \cdot t
\arrow[dr, Rightarrow, "\mu \cdot t"']
\arrow[r, draw=none, pos=0.5, "\threecell{0}"', "\alpha \cdot t"]
&
t \cdot t \cdot t
\arrow[r, Rightarrow, "t \cdot \mu" description]
\arrow[d, Rightarrow, "\mu \cdot t" description]
\arrow[dr, draw=none, pos=0.5, "\alpha \threecell{315} \phantom{\alpha}" description]
&
t \cdot t
\arrow[d, Rightarrow, "\mu"]
\arrow[r, draw=none, "=" description]
&
t \cdot t \cdot t
\arrow[r, Rightarrow, "t \cdot \mu" description]
\arrow[dr, Rightarrow, "\mu \cdot t"']
&
t \cdot t
\arrow[dr, Rightarrow, "\mu" description]
\arrow[r, draw=none, pos=0.5, "\alpha", "\threecell{0}"']
\arrow[d, draw=none, pos=0.5, "\alpha \threecell{270} \phantom{\alpha}" description]
&
t \cdot t
\arrow[d, Rightarrow, "\mu" description]
\\
&
t \cdot t
\arrow[r, Rightarrow, "\mu"']
&
t
&
&
t \cdot t
\arrow[r, Rightarrow, "\mu"']
&
t
\end{tikzcd}
\eqpunct{,}
\end{center}

\begin{center}
\begin{tikzcd}[row sep=huge, column sep={40pt}]
t \cdot t \cdot t
\arrow[r, Rightarrow, "t \cdot \mu"]
\arrow[dr, Rightarrow, "\mu \cdot t" description]
&
t \cdot t
\arrow[dr, Rightarrow, "\mu"{name=domequals}]
\arrow[d, draw=none, pos=0.5, "\alpha \threecell{270} \phantom{\alpha}" description]
&
&
t \cdot t \cdot t
\arrow[r, Rightarrow, "t \cdot \mu"]
\arrow[dr, draw=none, pos=0.2, "t \cdot \rho \threecell{315}" description]
&
t \cdot t
\arrow[dr, Rightarrow, "\mu"]
\arrow[d, draw=none, pos=0.5, "=" description]
&
\\
t \cdot t
\arrow[r, Rightarrow, "\mathrm{id}_{t \cdot t}"']
\arrow[u, Rightarrow, "t \cdot \eta \cdot t"]
\arrow[ur, draw=none, pos=0.2, "\lambda \cdot t \threecell{293}" description]
&
t \cdot t
\arrow[r, Rightarrow, "\mu"']
&
t
&
t \cdot t
\arrow[u, Rightarrow, "t \cdot \eta \cdot t"{name=codequals}]
\arrow[from=domequals, to=codequals, draw=none, pos=0.67, "=" description]
\arrow[rr, Rightarrow, "\mu"']
\arrow[ur, Rightarrow, "\mathrm{id}_{t \cdot t}" description]
&
~
&
t
\end{tikzcd}
\eqpunct{.}
\end{center}
\end{definition}

\begin{remark}
\label{rmk2MonadPseumonadOn2Cat}
We reserve the following terminology for particular cases of pseudomonads in $\mathbf{2Cat}$-enriched bicategories:
\begin{itemize}
\item When the 3-cells $\alpha,\lambda,\rho$ are identities, we call $\mathbb{T}$ a \textbf{2-monad} in $\mathfrak{B}$. Note that a 2-monad in $\mathfrak{B}$ restricts to a monad in the underlying bicategory of $\mathfrak{B}$, and that every monad in the underlying bicategory of $\mathfrak{B}$ is automatically a 2-monad in $\mathfrak{B}$.
\item When $\mathfrak{B}=\mathbf{2Cat}$ is the 3-category of 2-categories, 2-functors, pseudo-natural transformations and modifications, and the underlying 0-cell of $\mathbb{T}$ is a 2-category $\mathcal{K}$, we say that $\mathbb{T}$ is a pseudomonad (or 2-monad) \textbf{on} $\mathcal{K}$.
\end{itemize}
\end{remark}

\begin{definition}
\label{defPolynomialPseudomonad}
A \textbf{polynomial 2-monad} (resp.\ \textbf{polynomial pseudomonad}) is a 2-monad (resp.\ pseudomonad) in the $\mathbf{2Cat}$-enriched bicategory $\mathfrak{Poly}\cart_{\mathcal{E}}$. Specifically, a polynomial pseudomonad consists of the following data:
\begin{itemize}
\item An object $I$ of $\mathcal{E}$;
\item A polynomial $p : I \pto I$;
\item Cartesian morphisms of polynomials $\eta : i_I \pRightarrow p$ and $\mu : p \cdot p \pRightarrow p$;
\item Invertible adjustments $\alpha : \mu \circ (p \cdot \mu) \pRrightarrow \mu \circ (\mu \cdot p)$, $\lambda : \mu \circ (\eta \cdot p) \pRrightarrow \mathrm{id}_p$ and $\rho : \mu \circ (p \cdot \eta) \pRrightarrow \mathrm{id}_p$;
\end{itemize}
such that the adjustments $\alpha, \lambda, \rho$ satisfy the coherence axioms of Definition \ref{defPseudomonadIn2CatEnrichedBicategory}.
\end{definition}

A consequence of Theorem \ref{thmPolyECartTrivial} is that all parallel pairs of cartesian morphisms of polynomials are uniquely isomorphic. It follows that, in this case, simply specifying the \textit{data} for a polynomial monad suffices for defining a polynomial pseudomonad---this is stated precisely in the following lemma, whose proof is immediate.

\begin{lemma}
\label{lemMonadDataIsPseudomonad}
Let $I$ be an object of $\mathcal{E}$, let $p : I \pto I$ be a polynomial and let $\eta : i_I \pRightarrow p$ and $\mu : p \cdot p \pRightarrow p$ be cartesian morphisms of polynomials. Then there are unique adjustments $\alpha,\lambda,\rho$ such that the septuple $\mathbb{P}=(I,p,\eta,\mu,\alpha,\lambda,\rho)$ is a polynomial pseudomonad in $\mathcal{E}$. \qed
\end{lemma}

The next result allows us to lift polynomial 2-monads and polynomial pseudomonads \textit{in} $\mathcal{E}$ to 2-monads and pseudomonads \textit{on} the hom 2-categories of $\mathfrak{Poly}\cart_{\mathcal{E}}$. This will be key in Section~\ref{secMonadAlgebra} for identifying the sense in which a natural model $p : \dot{\mathcal{U}} \to \mathcal{U}$ is a pseudoalgebra over the polynomial pseudomonad it induces.

\begin{theorem}
\label{thmPolynomialPseudomonadLifts}
Let $\mathbb{P} = (p,\eta,\mu,\alpha,\lambda,\rho)$ be a polynomial 2-monad (resp.\ pseudomonad) on an object $I$ of a locally cartesian closed category $\mathcal{E}$. Then $\mathbb{P}$ lifts to a 2-monad (resp.\ pseudomonad) $\mathbb{P}^+ = (P, h, m, \dots)$ on $\mathfrak{Poly}\cart_{\mathcal{E}}(I,I)$.
\end{theorem}

\begin{proof}
By Lemma \ref{lemPolynomialsFromOneToOneTwoCategorical}, we may take $I=1$ without loss of generality, so that $p$ is just a morphism $p : Y \to X$ in $\mathcal{E}$ and $\eta,\mu$ are pullback squares in $\mathcal{E}$ (cf.\ Remark \ref{rmkCartesianMorphismIsPullbackSquare}).

For notational simplicity, write $\mathcal{K}$ to denote the 2-category $\mathfrak{Poly}\cart_{\mathcal{E}}(1,1)$. Note $\mathcal{K}$ has as its underlying category the wide subcategory $\mathcal{E}^{\to}_{\text{cart}}$ of $\mathcal{E}^{\to}$ whose morphisms are the pullback squares. Thus the 0-cells of $\mathcal{K}$ are the morphisms of $\mathcal{E}$, the 1-cells of $\mathcal{K}$ are pullback squares in $\mathcal{E}$, and between any two 1-cells there is a unique 2-cell by Theorem \ref{thmPolyECartTrivial}.

First we must define a 2-functor $P : \mathcal{K} \to \mathcal{K}$. Define $P$ on the 0-cells of $\mathcal{K}$ by letting $P(f)=P_p(f)$ for all $f : B \to A$ in $\mathcal{E}$. Given a 1-cell $\varphi : f \pRightarrow g$ of $\mathcal{K}$---that is, a pullback square in $\mathcal{E}$---let $P(\varphi)$ be the result of applying the extension $P_p$ of $p$ to the pullback square defining $\varphi$, as in:
\begin{center}
\begin{tikzcd}[row sep=huge, column sep=huge]
\sum_{x \in X} B^{Y_x}
\arrow[r, "P_p(\varphi_1)"]
\arrow[d, "P_p(f)"']
\pullbackc{dr}{-0.05}
&
\sum_{x \in X} D^{Y_x}
\arrow[d, "P_p(g)"]
\\
\sum_{x \in X} A^{Y_x}
\arrow[r, "P_p(\varphi_0)"']
&
\sum_{x \in X} C^{Y_x}
\end{tikzcd}
\eqpunct{.}
\end{center}
Note that $P(\varphi)$ is indeed a pullback square, since polynomial functors preserve all connected limits \cite{GambinoKock2013}. Thus $P(\varphi)$ is a 1-cell from $P(f)$ to $P(g)$ in $\mathcal{K}$.

Now $P$ respects identity 1-cells in $\mathcal{K}$, since if $f : B \to A$ is a 0-cell then
$$P(\mathrm{id}_f)_0 = P_p(\mathrm{id}_B) = \mathrm{id}_{P_p(B)} = (\mathrm{id}_{P(f)})_0
\eqpunct{,}$$
and likewise $P(\mathrm{id}_f)_1 = (\mathrm{id}_{P(f)})_1$; and $P$ respects composition of 2-cells in $\mathcal{K}$, since for $i \in \{0,1\}$ we have
$$P(\psi \circ \varphi)_i = P_p((\psi \circ \varphi)_i) = P_p(\psi_i \circ \varphi_i) = P_p(\psi_i) \circ P_p(\varphi_i) = P(\psi)_i \circ P(\varphi)_i = (P(\psi) \circ P(\varphi))_i
\eqpunct{.}$$
Hence the action of $P$ defines a functor on the underlying category of $\mathcal{K}$.

The fact that $P$ extends to a $2$-functor is trivial: given an adjustment $\alpha : \varphi \pRrightarrow \psi$, there is a unique adjustment $P(\varphi) \pRrightarrow P(\psi)$. We take this to be $P(\alpha)$, and note that the axioms governing identity and composition of 2-cells hold trivially by uniqueness of adjustments.

The pseudo-natural transformations $h : \mathrm{id}_{\mathcal{K}} \Rightarrow P$ and $m : P \circ P \Rightarrow P$ giving the unit and multiplication of $\mathbb{P}^+$ are induced by the unit $\eta : i_1 \pRightarrow p$ and $\mu : p \cdot p \pRightarrow p$ of $\mathbb{P}$. Specifically, define the components $h_f : f \pRightarrow P(f)$ and $m_f : P(P(f)) \pRightarrow P(f)$ at a 0-cell $f : B \to A$ of $\mathcal{K}$ to be the following squares, respectively:
\begin{center}
\begin{tikzcd}[row sep=normal, column sep=huge]
B
\arrow[r, "(P_{\eta})_B"]
\arrow[d, "f"']
\pullback
&
\sum_{x \in X} B^{Y_x}
\arrow[d, "P(f)"]
&
\sum_{(x,t) \in \sum_{x \in X} X^{Y_x}} B^{\left(\sum_{y \in Y_x} Y_{t(y)}\right)}
\arrow[r, "(P_{\mu})_B"]
\arrow[d, "P(P(f))"']
\pullbackc{dr}{-0.2}
&
\sum_{x \in X} B^{Y_x}
\arrow[d, "P(f)"]
\\
A
\arrow[r, "(P_{\eta})_A"']
&
\sum_{x \in X} A^{Y_x}
&
\sum_{(x,t) \in \sum_{x \in X} X^{Y_x}} A^{\left(\sum_{y \in Y_x} Y_{t(y)}\right)}
\arrow[r, "(P_{\mu})_A"']
&
\sum_{x \in X} A^{Y_x}
\end{tikzcd}
\eqpunct{.}
\end{center}
Note that these squares commute and are cartesian by naturality and cartesianness of the extensions $P_{\eta},P_{\mu}$ of $\eta,\mu$. That $h$ and $m$ extend to pseudo-natural transformations is immediate from Theorem \ref{thmPolyECartTrivial}: the pseudo-naturality 2-cells in $\mathcal{K}$ are adjustments, so they exist uniquely and satisfy the coherence axioms for pseudo-natural transformations automatically.

If $\mathbb{P}$ is a polynomial 2-monad, it is now easy to verify that the 2-monad laws hold for $\mathbb{P}^+$. If $\mathbb{P}$ is a polynomial pseudomonad, then the pseudomonad laws for $\mathbb{P}^+$ concern existence of and equations between adjustments, hence are trivially true by Theorem \ref{thmPolyECartTrivial}. 
\end{proof}

\begin{definition}
\label{defLiftOfPseudomonad}
Given a polynomial monad (resp.\ pseudomonad) $\mathbb{P}$, the \textbf{lift} of $\mathbb{P}$ is the 2-monad (resp.\ pseudomonad) $\mathbb{P}^+$ as in Theorem \ref{thmPolynomialPseudomonadLifts}.
\end{definition}

\begin{definition}
\label{defPseudoalgebra}
Let $\mathbb{T} = (T, h, m, \alpha, \lambda, \rho)$ be a pseudomonad on a 2-category $\mathcal{K}$. A \textbf{pseudoalgebra} over $\mathbb{T}$ consists of
\begin{itemize}
\item A 0-cell $A$ of $\mathcal{K}$;
\item A 1-cell $a : T(A) \to A$ in $\mathcal{K}$;
\item Invertible 2-cells $\sigma, \tau$ of $\mathcal{K}$, as in:
\begin{center}
\begin{tikzcd}[row sep=huge, column sep=huge]
T(T(A))
\arrow[r, "T(a)"]
\arrow[d, "m_T"']
&
T(A)
\arrow[d, "a"]
\arrow[dl, draw=none, pos=0.5, "\sigma \twocell{225}" description]
&
A
\arrow[r, "h_A"]
\arrow[dr, "\mathrm{id}_A"']
&
T(A)
\arrow[d, "a"]
\arrow[dl, draw=none, pos=0.25, "\tau \twocell{225}" description]
\\
T(A)
\arrow[r, "a"']
&
A
&
~
&
A
\end{tikzcd}
\eqpunct{,}
\end{center}
\end{itemize}
such that the following equations of pasting diagrams hold:

\begin{center}
\begin{tikzcd}[row sep=huge, column sep={40pt}]
T^3A
\arrow[r, "TTa"]
\arrow[d, "m_{TA}"']
\arrow[dr, "Tm_A" description]
&
T^2A
\arrow[dr, "Ta"]
\arrow[d, draw=none, pos=0.5, "T\sigma\ \twocell{270} \phantom{T\sigma}" description]
&
&
T^3A
\arrow[r, "T^2a"]
\arrow[d, "m_{TA}"']
\arrow[dr, draw=none, pos=0.5, "\cong" description]
&
T^2A
\arrow[dr, "Ta"]
\arrow[d, "m_A" description]
&
\\
T^2A
\arrow[dr, "m_A"']
\arrow[r, draw=none, pos=0.5, "\twocell{0}"', "\alpha_A"]
&
T^2A
\arrow[r, "Ta" description]
\arrow[d, "m_A" description]
\arrow[dr, draw=none, pos=0.5, "\sigma \twocell{315} \phantom{\sigma}" description]
&
TA
\arrow[d, "a"]
\arrow[r, draw=none, "=" description]
&
T^2A
\arrow[r, "Ta" description]
\arrow[dr, "m_A"']
&
TA
\arrow[dr, "a" description]
\arrow[r, draw=none, pos=0.5, "\sigma", "\twocell{0}"']
\arrow[d, draw=none, pos=0.5, "\sigma \twocell{270} \phantom{\sigma}" description]
&
TA
\arrow[d, "a" description]
\\
&
TA
\arrow[r, "a"']
&
A
&
&
T(A)
\arrow[r, "a"']
&
A
\end{tikzcd}
\eqpunct{,}
\end{center}

\begin{center}
\begin{tikzcd}[row sep=huge, column sep={40pt}]
T^2A
\arrow[r, "Ta"]
\arrow[dr, "m_A" description]
&
TA
\arrow[dr, "a"{name=domequals}]
\arrow[d, draw=none, pos=0.5, "\sigma \twocell{270} \phantom{\sigma}" description]
&
&
T^2A
\arrow[r, "Ta"]
\arrow[dr, draw=none, pos=0.2, "T\tau \twocell{315}" description]
&
TA
\arrow[dr, "a"]
\arrow[d, draw=none, pos=0.5, "=" description]
&
\\
TA
\arrow[r, "\mathrm{id}_{TA}"']
\arrow[u, "Th_A"]
\arrow[ur, draw=none, pos=0.2, "\lambda_A \twocell{293} \phantom{\lambda_A}" description]
&
TA
\arrow[r, "a"']
&
TA
&
TA
\arrow[u, "Th_A"{name=codequals}]
\arrow[from=domequals, to=codequals, draw=none, pos=0.67, "=" description]
\arrow[rr, "a"']
\arrow[ur, "\mathrm{id}_{TA}" description]
&
~
&
A
\end{tikzcd}
\eqpunct{.}
\end{center}
\end{definition}

\begin{definition}
\label{defPolynomialPseudoalgebra}
Let $\mathbb{P} = (1, p : Y \to X, \dots)$ be a polynomial pseudomonad in a locally cartesian closed category $\mathcal{E}$. A \textbf{polynomial pseudoalgebra} over $\mathbb{P}$ is a pseudoalgebra over the lift $\mathbb{P}^+$. Specifically, it consists of:
\begin{itemize}
\item A polynomial $f : B \to A$;
\item A cartesian morphism of polynomials $\zeta : P_p(f) \pRightarrow f$;
\item Invertible adjustments $\sigma,\tau$ whose types are as in Definition \ref{defPseudoalgebra};
\end{itemize}
such that the adjustments $\sigma,\tau$ satisfy the coherence conditions of Definition \ref{defPseudoalgebra}.
\end{definition}

Much like with polynomial pseudomonads (Lemma \ref{lemMonadDataIsPseudomonad}), merely specifying the \textit{data} for a polynomial pseudoalgebra suffices for the conditions to hold---again, this follows immediately from Theorem \ref{thmPolyECartTrivial}.

\begin{lemma}
\label{lemAlgebraDataIsPseudoalgebra}
Let $\mathbb{P} = (I, p : Y \to X, \dots)$ be a polynomial pseudomonad in a locally cartesian closed category $\mathcal{E}$, let $f : B \to A$ be a polynomial and let $\zeta : P_p(f) \pRightarrow f$ be a morphism of polynomials. Then there are unique adjustments $\sigma,\tau$ making $(f,\zeta,\sigma,\tau)$ into a polynomial pseudoalgebra over $\mathbb{P}$. \qed
\end{lemma}

\newpage
\section{Type theory is a pseudomonad and a pseudoalgebra}
\label{secMonadAlgebra}

The results of Section \ref{secTricategory} allow us to precisely formulate and easily prove the conjecture outlined in Section \ref{secNaturalModels}.

\begin{theorem}
\label{thmUnitSigmaIffPolynomialPseudomonad}
Let $\mathsf{C} = (\mathbb{C},p)$ be a natural model.
\begin{enumabc}
\item $\mathsf{C}$ supports a unit type and dependent sum types if and only if $p$ can be equipped with the structure of a polynomial pseudomonad $\mathbb{P}$ in $\widehat{\mathbb{C}}$.
\item $\mathsf{C}$ additionally supports dependent product types if and only if $p$ can be equipped with the structure of a polynomial pseudoalgebra over $\mathbb{P}$.
\end{enumabc}
\end{theorem}
\begin{proof}
By Theorem \ref{thmUnitSigmaPiPoly}, $\mathsf{C}$ supports a unit type and dependent sum types if and only if there exist cartesian morphisms of polynomials $\eta : i_1 \pRightarrow p$ and $\mu : p \cdot p \pRightarrow p$, and additionally supports dependent product types if and only if there exists a cartesian morphism of polynomials $\zeta : P_p(p) \pRightarrow p$. By Lemmas \ref{lemMonadDataIsPseudomonad} and \ref{lemAlgebraDataIsPseudoalgebra}, there are unique adjustments turning $(p,\eta,\mu)$ into a polynomial pseudomonad $\mathbb{P}$, and unique adjustments turning $(p,\zeta)$ into a polynomial pseudoalgebra over $\mathbb{P}$.
\end{proof}

\begin{remark}
Theorem \ref{thmUnitSigmaIffPolynomialPseudomonad} makes a connection between logic and algebra by exhibiting a correspondence between laws concerning dependent sums and dependent products in type theory with laws concerning monads in algebra. Specifically, for $\eta : \iota_1 \pRightarrow p$, $\mu : p \cdot p \pRightarrow p$ and $\zeta : P_p(p) \pRightarrow p$, the (pseudo)monad and (pseudo)algebra equations correspond to certain type isomorphisms as follows:
\begin{center}
\begin{tabular}{c|c}
\textbf{Monads and algebras} & \textbf{Type theory} \\ \hline
$\mu \circ (p \cdot \mu) \cong \mu \circ (\mu \cdot p)$ & $\sum_{x:A} \sum_{y:B(x)} C(x,y) \cong \sum_{\langle x,y \rangle : \sum_{x:A} B(x)} C(x,y)$ \\
& \\
$\mu \circ (p \cdot \eta) \cong \mathrm{id}_p$ & $\sum_{x:A} \mathbf{1} \cong A$ \\
& \\
$\mu \circ (\eta \cdot p) \cong \mathrm{id}_p$ & $\sum_{x:\mathbf{1}} A \cong A$ \\
& \\
$\zeta \circ (p \cdot \zeta) \cong \zeta \circ (\mu \cdot p)$ & $\prod_{x:A} \prod_{y : B(x)} C(x,y) \cong \prod_{\langle x,y \rangle : \sum_{x:A} B(x)} C(x,y)$ \\
& \\
$\zeta \circ (\eta \cdot p) \cong \mathrm{id}_p$ & $\prod_{x:\mathbf{1}} A \cong A$
\end{tabular}
\end{center}
\end{remark}

\newpage
\section{Acknowledgements}

Nicola Gambino hosted the second-named author at the University of Leeds in early 2017, and was very helpful in catalysing our progress; Jonas Frey has also provided a great deal of useful input since his arrival at Carnegie Mellon in the summer of 2016. We are grateful for numerous discussions with both Jonas and Nicola.

Since our first public presentation of this material in March 2017 \cite{AwodeyNewstead2017}, we have discovered that Thorsten Altenkirch and Gun Pinyo have independently presented related results (under the name \textit{monadic containers}) at the TYPES conference that took place in May--June 2017 in Budapest \cite{AltenkirchPinyo2017}.

We gratefully acknowledge the support of the Air Force Office of Scientific Research through MURI grant FA9550-15-1-0053. Any opinions, findings and conclusions or recommendations expressed in this material are those of the authors and do not necessarily reflect the views of the AFOSR.

\bibliographystyle{alpha}
\bibliography{pseudomonads}
\addcontentsline{toc}{section}{References}
\nocite{*}

\end{document}